\documentclass{article}
\usepackage{graphicx} % Required for inserting images

\usepackage[colorinlistoftodos]{todonotes}
\usepackage{pgfplots} 
\usepackage{subfigure}
\usepackage{amsmath}
\usepackage{amssymb} 
\usepackage{mathtools}   % loads »amsmath«
\usepackage{graphicx}
\usepackage{graphics}  
\usepackage{color}
\usepackage{verbatim} 
\usepackage{bbold}
\usepackage[normalem]{ulem} 
\usepackage[mathscr]{eucal}
\usepackage{upgreek}
\usepackage{enumerate} 
\usepackage{cite}
\usepackage{amsmath}
\usepackage{amsfonts}
\usepackage{amssymb}
\usepackage{bbm}
\usepackage{cancel}
\usepackage{graphicx}
\usepackage{multicol}
\usepackage[utf8]{inputenc}
\usepackage{framed}
\usepackage{setspace}
\usepackage{amsthm}
\usepackage{hyperref}
\usepackage{caption}
\usepackage{subcaption}
\usepackage{bbm}
\usepackage{graphicx}
\usepackage{tikz}
\usepackage[titletoc]{appendix}
\numberwithin{equation}{section}  

\usepackage[refpage,noprefix]{nomencl}
\usepackage{nomencl} 

\usetikzlibrary{arrows.meta, positioning}

\newcommand{\gk}[1]{\left\{#1\right\}}

\newcommand\mycom[2]{\genfrac{}{}{0pt}{}{#1}{#2}}
\newcommand{\ek}[1]{\left[#1\right]}
\newcommand{\rk}[1]{\left(#1\right)}

\newcommand{\abs}[1]{\left| #1 \right|}

\newcommand{\Ocal}   {{\mathcal O }}

\newcommand{\R}     {\mathbb{R}}

\renewcommand{\P}   {\mathbb{P}} 
 
\newcommand{\E}     {\mathbb{E}}

 \newcommand{\ex}{{\rm e}} 
 \renewcommand{\d}{{\rm d}}

\renewcommand{\l}{\lambda}
\newcommand{\e}{\varepsilon}
\newcommand{\1}{\mathbbm{1}}

\renewcommand{\P}{\mathbb{P}}

\newtheorem*{remark}{Remark}

\newtheorem{theorem}{Theorem}
\newtheorem{lemma}{Lemma}[section]
\newtheorem{definition}[lemma]{Definition}

\newtheorem{corollary}[lemma]{Corollary}

\newcommand{\Var}{\mathrm{Var}}

\title{Phase Transitions in a Particle Model\\ for the Self-Adaptive Response to Cancer Dynamics}
\author{Christian Kuehn$^{1}$ and Quirin Vogel$^{2}$}
\date{\today}

\begin{document}

\maketitle

\begin{abstract}
    In this paper, we present a probabilistic analysis of a dynamical particle model for the self-adaptive immune response to cancer, as proposed in~\cite{kuehn2025mathematicalmodelsselfadaptiveresponse}. The model is motivated by the interplay between immune surveillance and cancer evolution. We rigorously confirm the sharp phase transition in immune system \textit{learning} predicted in the original work. Additionally, we compute the expected amount of information acquired by the immune system about cancer cells over time. Our analysis relies on time-reversal techniques.
\end{abstract}
\centerline{\textit{$^1$Technical University of Munich, Garching, Germany}}
\centerline{\textit{$^2$
Alpen-Adria-Universität Klagenfurt, Austria}}
\bigskip


\section{Introduction}
This paper extends the work of~\cite{kuehn2025mathematicalmodelsselfadaptiveresponse}, which introduced two models for the self-adaptive response to cancer dynamics. The first is a macroscopic model describing adaptivity via an ordinary differential equation (ODE), presented in Section 2 of the referenced paper. The second is a microscopic stochastic model, supported by simulations in~\cite[Section 3]{kuehn2025mathematicalmodelsselfadaptiveresponse}. In the present work, we provide a rigorous mathematical analysis of the stochastic model, complementing the previous simulations with formal proofs.

The primary motivation for the microscopic model is to capture the competition between cancer-promoting mechanisms (e.g., adaptive mutations in cancer cells) and cancer-inhibiting processes (e.g., immune responses). Cancer cells gain an advantage when they successfully evade immune detection. Conversely, it is widely conjectured that immunotherapy and other immune-mediated mechanisms may be able to effectively counteract tumor growth in many cases~\cite{baxevanis2023thin,khanmammadova2023neuro}.

We briefly describe the model; see Section~\ref{sec:model} for the complete formulation. Suppose that there are $N\ge 1$ distinct components that the immune system can use to combat cancer cells. Each component has $M\ge 1$ attributes. In particular, we abstractly model that the immune system gradually acquires knowledge and adapts to new diseases~\cite{ParkinCohen}, e.g., trying to eliminate cancer cells. We refer to this learning mechanism as the process of adding information (PAI). The information level for each attribute is represented by an integer in $\gk{0,\ldots, M}$, where $0$ indicates no knowledge and $M$ indicates complete knowledge. A cancer cell can be effectively targeted only once all its attributes have been fully learned.

This immune response is counteracted by cancer cell mutation, which can make previously acquired knowledge obsolete. Simulations in~\cite{kuehn2025mathematicalmodelsselfadaptiveresponse} suggest the emergence of a sharp phase transition: given sufficient time and given sufficient external input regarding the main characteristics of cancer cells, the immune system rapidly learns a positive fraction of all components around a critical time $T_{\mathrm{pt}}$.

The contributions of this paper are two-fold:
\begin{enumerate}
\item We calculate the expected transition time $T_{\mathrm{pt}}$, and show that its fluctuations occur on a smaller scale than its mean---demonstrating sharpness of the transition.
\item We derive an explicit expression for the expected number of learned components in the steady state.
\end{enumerate}

We conduct our analysis in two settings:
\begin{enumerate}
\item Considering the aggregate behavior of all $N$ components.
\item Tracking the dynamics of a single isolated component.
\end{enumerate}
The results differ between these two settings, for example, the scale of $T_{\mathrm{pt}}$ and the size of the transition window, due to the interaction structure of the stochastic cancer-immune dynamics.

Our proof strategy differs between the two models. For the evolution of a single component, we use the fact that the invariant distribution can be expressed using ratios of Gamma functions. For the sharpness of the phase transition, we use a second-moment argument, by deriving a recursive equation for an upper bound of the variance of the transition time. For the analysis of the full model, we can use some results on the coupon collector problem (in random-time) to lower bound the transition time $T_{\mathrm{pt}}$. For the upper-bound, we derive an algorithm for sampling the invariant distribution, which limits $T_{\mathrm{pt}}$ from above and use time-reversal, see \cite[Chapter 10]{chung2005markov} for an introduction to this method. 

While our investigation settles the question of sharpness and steady-state analysis, a number of open questions remain. For example, it would be interesting to see how to fit the parameters in our model from observations. This is relevant because it would allow us to calculate the necessary time for the immune response to adapt to the dynamics of cancer and potentially influence treatment strategies. Another question is how universal with respect to the dynamics the sharpness of the transition is? Since our model is \textit{non-reversible}, a standard analysis of sharpness using eigenvalues is not readily available (see~\cite{levin2017markov}), although recently some general results on bounds have been attained for non-reversible chain, see~\cite{chatterjee2023spectral}. Finally, studying the invariant law in more detail and proving convergence under suitable rescaling remains an open question.

Our paper is structured as follows: in Section~\ref{sec:model}, we define both our stochastic models and relate them to~\cite{kuehn2025mathematicalmodelsselfadaptiveresponse}. In Section~\ref{sec:results}, we give the results. In Section~\ref{sec:proof}, we prove the statements, starting with the evolution of a single attribute, splitting the proof of the multicolumn model in two cases.
%%%%
\section{The model}\label{sec:model}
Our model is based on the paper~\cite{kuehn2025mathematicalmodelsselfadaptiveresponse}. However, we choose to state our results in continuous-time as opposed to the discrete-time setting in~\cite{kuehn2025mathematicalmodelsselfadaptiveresponse}. This models the real-world behavior of stochastic processes, where transitions occur at random times.
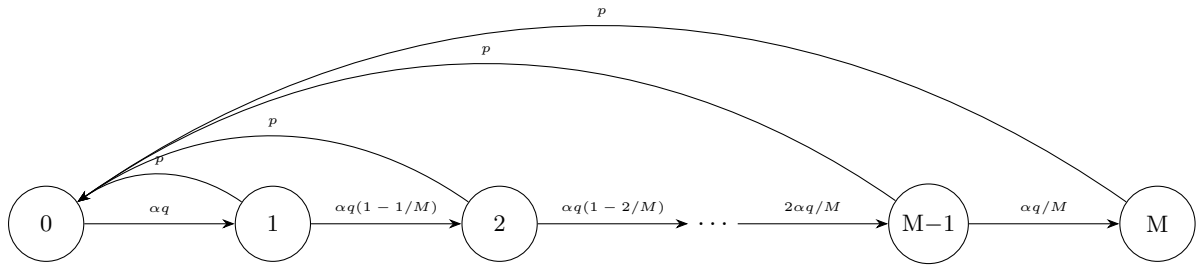
\begin{figure}
    \centering
    \begin{tikzpicture}[
    >=Stealth,
    node distance = 2cm and 2cm,
    state/.style = {circle, draw, minimum size=1cm, font=\small},
    edge/.style  = {->, >=Stealth}
  ]

% Nodes
\node[state] (0) {0};
\node[state, right=of 0] (1) {1};
\node[state, right=of 1] (2) {2};
\node[draw=none, right=of 2] (dots) {$\ldots$};
\node[state, right=of dots] (m1) {M${-}1$};
\node[state, right=of m1]  (m) {M};

% Forward arrows with labels
\draw[edge] (0)  -- node[midway, above] {\tiny{$\alpha q$}} (1);
\draw[edge] (1)  -- node[midway, above] {\tiny{$\alpha q(1-1/M)$}} (2);
\draw[edge] (2)  -- node[midway, above] {\tiny{$\alpha q(1-2/M)$}} (dots);
\draw[edge] (dots) -- node[midway, above] {\tiny{$2\alpha q/M$}} (m1);
\draw[edge] (m1) -- node[midway, above] {\tiny{$\alpha q/M$}} (m);

% Backward arrows to 0 with labels
\draw[edge, bend right=35] (1) to node[midway, above] {\tiny{$p$}} (0);
\draw[edge, bend right=35] (2) to node[midway, above] {\tiny{$p$}} (0);
\draw[edge, bend right=35] (m1) to node[midway, above] {\tiny{$p$}} (0);
\draw[edge, bend right=35] (m)  to node[midway, above] {\tiny{$p$}} (0);

\end{tikzpicture}
    \caption{Illustration of chain from Definition~\ref{def:singleCol}}
    \label{fig:singleCol}
\end{figure}
First, let us introduce the evolution of a single column.
\begin{definition}\label{def:singleCol}
    Let $\alpha>0$ and $0<p=1-q<1$. Let $\rk{X_t}_{t\ge 0}$ be the continuous-time Markov process on $\gk{0,1,\ldots, M}$ specified by the following transition rates $Q\colon \gk{0,\ldots,M}^2\to\R$
\begin{equation}
    Q(k,j)=\begin{cases}
        \alpha q\rk{1-\tfrac{k}{M}}&\textnormal{ if }j=k+1\, ,\\
        p&\textnormal{ if }j=0\, ,\\
        0&\textnormal{ otherwise.}
    \end{cases}
\end{equation}
\end{definition}
See Figure~\ref{fig:singleCol} for an illustration. Note that the above chain is non-reversible!
\begin{remark}\label{rem:equiv}
    The model introduced in Definition \ref{def:singleCol} asymptotically tracks the evolution of a single, fixed component in the microscopic model introduced in~\cite{kuehn2025mathematicalmodelsselfadaptiveresponse} in continuous-time, when making the identification
    \begin{equation}
        p\mapsto \frac{p_d}{N}\quad\text{and}\quad \alpha=1+\l_m\, ,
    \end{equation}
    where we use the notation from the original paper and given the assumption that $p_m M$ converges to $\lambda_m>0$.
\end{remark}
The remark above is proven in Lemma \ref{lem:id}.

Next, we introduce the random matrix evolution, motivated by~\cite{kuehn2025mathematicalmodelsselfadaptiveresponse}. 
\begin{definition}\label{def:mutlCol}
    Let $(A_t)_{t\ge 0}$ with $A_t\in\R^{M\times N}$ be the continuous-time Markov chain with following transition rates
    \begin{equation}\label{eq:transitionRates}
        Q(A,B)=\begin{cases}
            \tfrac{q}{M}&\textnormal{ if }B=A\oplus_j\1\, ,\\
            \tfrac{p}{N}&\textnormal{ if }B=A\ominus_i\mathbb{0}\, ,\\
            \frac{\l_m}{M}&\textnormal{ if }B=A\oplus_{i,j}\1\, ,\\
            0&\textnormal{ otherwise,}
        \end{cases}
    \end{equation}
    where $A\oplus_j\1$ is the matrix where we replace the $j$-th row by $\1=(1,\ldots,1)$, $A\oplus_{i,j}\1$ replaces the $(i,j)$-th entry with 1 and $A\ominus_i\mathbb{0}$ replaces the $i$-th column with $(0,\ldots,0)^\mathrm{T}$. Assume w.l.o.g. that $p,q>0$, $\l_m\ge 0$ and $p+q=1$.
\end{definition}
\begin{figure}
\begin{tikzpicture}[scale=0.85, transform shape, >=Stealth, node distance=1cm, every node/.style={font=\small}]

% Original matrix A
\node (A) {
$
A = 
  \begin{bmatrix}
    \star & \dots & \star & \dots & \star \\
    \vdots& \ddots& \vdots & \ddots & \vdots \\
    \star & \dots & \star & \dots & \star \\
    \vdots& \ddots& \vdots & \ddots & \vdots \\
    \star  & \dots& \star & \dots & \star
  \end{bmatrix},
  $
};

% Matrix with row replaced by 1s
\node[right=of A, xshift=-1cm] (row1s) {
$
 A\oplus_j\1= \begin{bmatrix}
    \star & \dots & \star & \dots & \star \\
    \vdots&  & \vdots &  & \vdots \\
    \color{blue}1 & \dots & \color{blue}1 & \dots & \color{blue}1 \\
    \vdots&  & \vdots &  & \vdots \\
    \star & \dots & \star & \dots & \star
  \end{bmatrix},
$
};
\node[below=of row1s, yshift=1cm] (eqn) {Rate $q/M$};
% Matrix with column replaced by 0s
\node[right=of A,xshift=4cm] (col0s) {
  $A\ominus_i\mathbb{0}=
  \begin{bmatrix}
    \star & \dots & \color{red}0 & \dots & \star \\
    \vdots&  & \vdots &  & \vdots \\
    \star & \dots & \color{red}0 & \dots & \star \\
    \vdots&  & \vdots &  & \vdots \\
    \star & \dots & \color{red}0 & \dots & \star
  \end{bmatrix},
  $
};
\node[below=of col0s, yshift=1cm] (eqn) {Rate $p/N$};
% Matrix with a single flipped entry
\node[right=of A, xshift=9cm] (flip1) {
  $A\oplus_{i,j}\1=
  \begin{bmatrix}
    \star & \dots & \star & \dots & \star \\
    \vdots&  & \vdots &  & \vdots \\
    \star & \dots & \color{green}1 & \dots & \star \\
    \vdots&  & \vdots &  & \vdots \\
    \star & \dots & \star & \dots & \star
  \end{bmatrix}.
  $
};
\node[below=of flip1, yshift=1cm] (eqn) {Rate $\lambda_m/M$};
% Arrows with labels
% \draw[->] (A) -- (row1s) node[midway, above left] {$q$};
% \draw[->] (A) -- (flip1)  node[midway, above] {$p_m$};
% \draw[->] (A) -- (col0s)  node[midway, below left] {$p$};

\end{tikzpicture}
    \caption{Visualization of the matrix transitions in Definition \ref{def:mutlCol}.}
    \label{fig:trans}
\end{figure}
See Figure~\ref{fig:trans} for a visualization of the above model with its transition rates. From an applied perspective, the three processes in Definition~\ref{def:mutlCol} model active externally-driven targeted learning, forgetting (or cancer evasion) of learned structures, and positive evolutionary mutation respectively.

\begin{remark}
    The above Markov chain models the evolution of all components, as described in~\cite{kuehn2025mathematicalmodelsselfadaptiveresponse}, in continuous time when making the identification $p_m=\tfrac{\l_m}{M}$. The case of $p_m=0$ is referred to as PAI switched off and is easier to analyse.
\end{remark}

\section{Results}\label{sec:results}
First, we characterize the distribution of a single component.
\begin{theorem}[Evolution of a single column]\label{thm:singleCol}
    Given $\alpha>0$ the speed of the updates to 1 and $p$ the probability of deletion, as described in Definition \ref{def:singleCol}. Let $\tau$ be the first time the column consists of a row of only ones. Choose $a=\tfrac{p M}{\alpha q}$ and assume that $a\in (0,\infty)$ is fixed as $M\to\infty$. We then have that 
    \begin{equation}
        \E\ek{\tau|X_0=0}\sim  \frac{M^{a+1}}{\Gamma(a+1)a}\, ,
    \end{equation}
    and
    \begin{equation}
        \P\rk{\abs{\tau-\E\ek{\tau}}\gg M^{(a+1)/2}}=o(1)\, ,
    \end{equation}
    as $M\to\infty$. This means that we observe a cut-off phenomenon on the square-root scale.

    Finally, we also have the following transitions based on the value of $a$: let $\pi$ be the invariant distribution of the chain and let $\P_\pi$ be the distribution of the chain started from the invariant distribution. There then exists $C>0$ such that
    \begin{equation}
                C^{-1}k^{a-1}\le \frac{\P_\pi\rk{ k\textnormal{ zeros}}}{\P_\pi\rk{\textnormal{ no zeros}}}\le C k^{a-1}\, ,
        \end{equation}
        as $M\to\infty$.
        
        Hence 
    \begin{enumerate}
        \item if $a>1$: most of the times we have a large number of ones,
        \item if $a=1$, ones and zeros balance,
        \item if $a<1$, most of the times, there is a large number of zeros.
    \end{enumerate}
\end{theorem}
Next, we discuss the model where we delete columns and add rows/single entries.
\begin{theorem}\label{thm:multCol}
\begin{enumerate}
    \item(PAI off, $\l_m=0$)
    Let $(A_t)_{t\ge 0}$ be the evolution of the model according to Definition \ref{def:mutlCol}. Let $\tau$ be the first time a column is equal to $\1=(1,\ldots,1)$. Denote $b=\tfrac{pM}{qN}$ and assume that this remains constant. We then have that for $\e_M\to 0$ sufficiently slowly\footnote{here, $\e_M>\log^{-1/2}(M)$ is sufficient}
    \begin{equation}\label{eq:transTimePAIoff}
        \lim_{\mycom{M\to\infty}{b={pM}/\rk{qN}}}\P\rk{\abs{\tau-\frac{M\log(M)}{q}}\le \e_M M\log(M)\mid A_0\equiv 0}=1\, .
    \end{equation}
    Furthermore, as $M\to\infty $ and $b=\tfrac{pM}{qN}$ constant, we have that
    \begin{equation}\label{eq:RationPAIoff}
        \E_\pi\ek{\#\gk{\text{columns }c\text{ with }c\equiv 1}}\sim N\frac{\Gamma(b+1)}{ (M/q)^{b}}\, ,
    \end{equation}
    where $\pi$ is the stationary distribution (on the space of matrices). Furthermore, after time $\tau(1+\e)$, the number of columns equal to 1 agrees with the stationary distribution. This implies that the model has a cut-off phenomenon on the $\log$-scale.
    \item(PAI on) Introduce now the new parameter $\widetilde{q}=q+\lambda_m$. Denote $\widetilde{b}=\tfrac{{p}M}{\widetilde{q}N}$ and assume that it remains constant. We then have that under the same conditions on $\e_M$ as before
    \begin{equation}
         \lim_{\mycom{M\to\infty}{\widetilde{b}={pM}/\rk{qN}}}\P\rk{\abs{\tau-\frac{M\log(M)}{\widetilde{q}}}\le \e_M M\log(M)\mid A_0\equiv 0}=1\, ,
    \end{equation}
    as well as
    \begin{equation}
        \E_\pi\ek{\#\gk{\text{columns }c\text{ with }c\equiv 1}}\sim N\frac{\Gamma\rk{\widetilde{b}+1}}{ (M/\widetilde{q})^{\widetilde{b}}}\, .
    \end{equation}
    Hence, PAI can be achieved by a change of parameters.
\end{enumerate}
\end{theorem}
See Figure~\ref{fig:Kuehn1} and Figure~\ref{fig:Kuehn2} for an illustration of the theorem for the parameters used in~\cite{kuehn2025mathematicalmodelsselfadaptiveresponse}.
\begin{figure}
    \centering
    \includegraphics[width=0.85\linewidth]{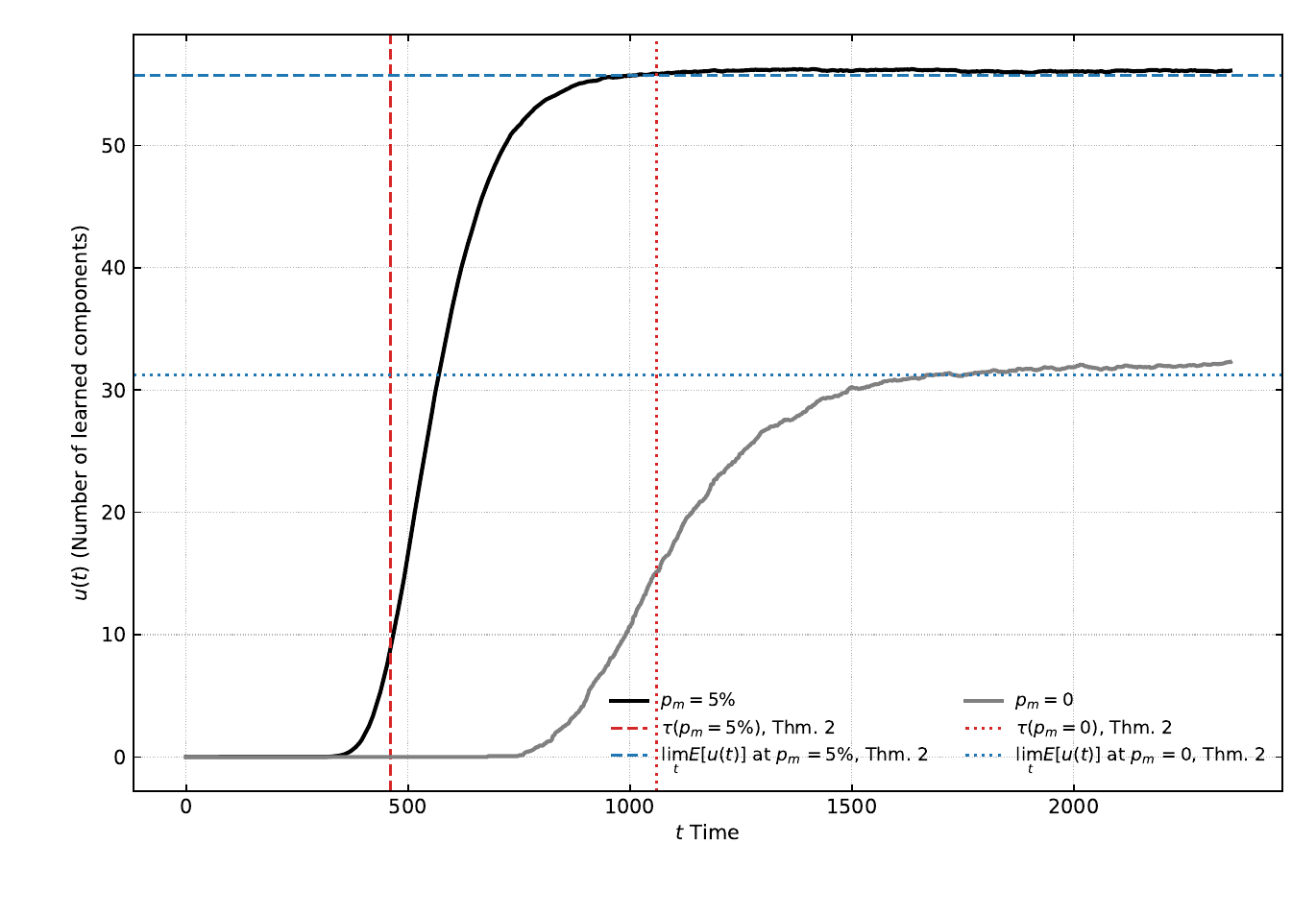}
    \caption{Simulations from~\cite[Figure 4]{kuehn2025mathematicalmodelsselfadaptiveresponse}, predicted values for transition time and proportion of ones from Theorem \ref{thm:multCol} as straight lines, $p_d=0.1$, $M=200$, $N=100$, $T=2500$, $p_m=0$ respectively $p_m=0.005$. 1000 samples path average.}
    \label{fig:Kuehn1}
\end{figure}
\begin{figure}
    \centering
    \includegraphics[width=0.5\linewidth]{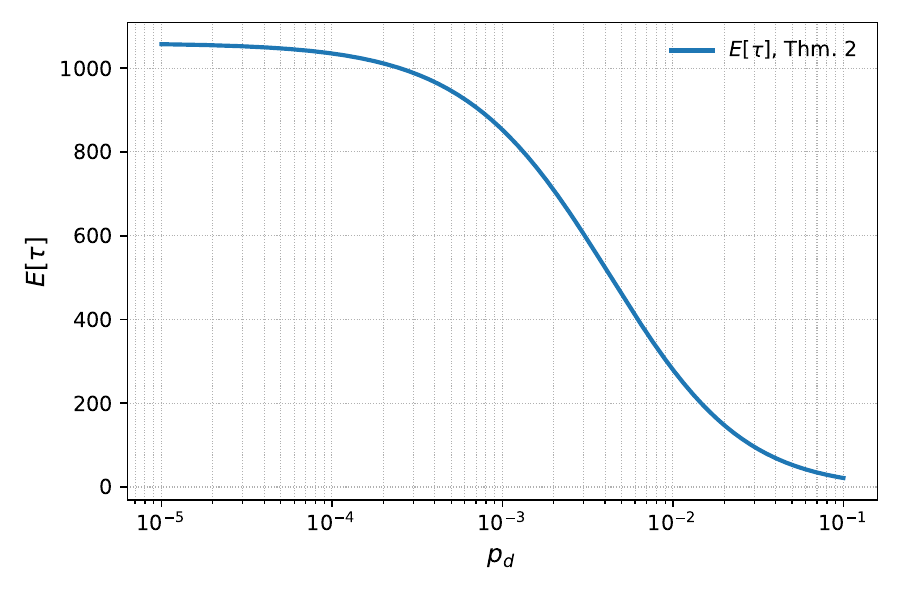}
    \caption{Plot of the expected transition time with $M = 200$, $N = 100$, and $p_d=0.1$. We plot the location of the transition point $\tau$ in time depending on the good mutation probability $p_m$. The scale is semi-logarithmic with a log-scale on the horizontal axis. Compare with~\cite[Figure 5]{kuehn2025mathematicalmodelsselfadaptiveresponse}, to see the match between prediction and simulation.}
    \label{fig:Kuehn2}
\end{figure}
\section{Proofs}\label{sec:proof}
This section is structured as follows: we begin with a proof of Theorem \ref{thm:singleCol}. We then write a more detailed proof of the first part of Theorem \ref{thm:multCol} and finally highlight the changes needed to accommodate PAI on.
\subsection{A single column}
The state of a single column can be modeled by a process on $\gk{0,\ldots, M}$, where the state indicated the number of entries equal to one in the systems.

We begin by making the connection between the model introduced in~\cite{kuehn2025mathematicalmodelsselfadaptiveresponse} and our set-up rigorous.
\begin{lemma}[Proof of Remark \ref{rem:equiv}]\label{lem:id}
    Given that $p_mM\to \lambda_m$, we have that the total variation distance between between the two models converges to zero. 
\end{lemma}
\begin{proof}
    Indeed, rule $(S1a)$ from~\cite{kuehn2025mathematicalmodelsselfadaptiveresponse} turns a uniformly at random selected row into ones. Hence, the probability that if $k$ out of the $M$ rows are occupied, to select an unoccupied row is given by $1-\tfrac{k}{M}$. This shows that the transition rate $Q(k, k+1)$ should be proportional to $(1-\tfrac{k}{M})$.
    
    Furthermore, by the Poisson--Binomial limit theorem, see~\cite{kallenberg1997foundations}, the evolution $(S1c)$ from the original paper converges to a Poisson process for each individual column, when $p_mM\to \lambda_M$.
\end{proof}
We start by giving the exact distribution of mass associated to each state for the invariant measure.
\begin{lemma}\label{lem:invardist}
    Let $\alpha$ the rate at which we add ones to the system, i.e.,  for $k\in \gk{0, \ldots,M}$
    \begin{equation}\label{eq:rates}
    Q(k,k+1)=\alpha q(1-\tfrac{k}{M})\quad\text{and}\quad Q(k,0)=p\1\gk{k>0}-\alpha q\1\gk{k=0}\, .
\end{equation}
We then have for the invariant distribution $\rk{\pi_k}_{k=0}^M$
\begin{equation}\label{eq:invariantDist}
    \pi_k=\frac{\Gamma(M+1)}{\Gamma(M+1-k)}\frac{\Gamma(\beta+ M-k)}{\Gamma(\beta+ M)}\frac{p}{p+\alpha q}\, ,
\end{equation}
where $\beta=pM/(\alpha q)$.
\end{lemma}
\begin{proof}
Recall that the invariant distribution satisfies $\pi Q\equiv 0$, see~\cite{liggett1985interacting} and that $Q(k,k)=-\sum_{j\neq k}Q(k,j)$. Eq.~\eqref{eq:rates} implies that to find the invariant law, we need to solve the system of $M+1$ equations
\begin{equation}
    \begin{split}
        \mathrm{0}.)&\, -\pi_0\alpha q+(1-\pi_0)p=0\, ,\\
        {k}.)&\, \pi_{k-1}\alpha q\rk{1-\tfrac{k-1}{M}}-\pi_k\rk{p+\alpha q\rk{1-\tfrac{k}{M}}}=0\, , \quad \textnormal{for }k\in \gk{1,\ldots, M}\, .
    \end{split}
\end{equation}
Equation $\mathrm{0}.)$ immediately gives $\pi_0=\tfrac{p}{p+\alpha q}$, while the other $k$ equations specify the ratio of $r_k=\pi_k/\pi_{k-1}=\tfrac{M-(k-1)}{\beta+1-k}$. This gives the result by writing
\begin{equation}
    \pi_k=\pi_0\prod_{i=1}^{k}r_i=\frac{p}{p+\alpha q}\prod_{i=1}^{k}\frac{M-(k-1)}{\beta+1-k}\, ,
\end{equation}
and using the fundamental recursion for the Gamma function.
\end{proof}
Note that the jump-chain of the model from Lemma~\ref{lem:invardist} is equivalent to a discrete time Markov chain with updated parameters
\begin{equation}
    p\mapsto \frac{p}{\alpha q+p}\quad\textnormal{and}\quad q\mapsto \frac{q}{\alpha q+p}\, ,
\end{equation}
see e.g.~\cite{levin2017markov}. Hence, calculating the expected hitting time and its variance, there is no loss of generality when switching to the discrete model. We therefore assume now transition probabilities
\begin{equation}
    p(k,0)=p\quad\textnormal{and}\quad p(k,k)=q\frac{k}{M}\quad\textnormal{and}\quad p(k,k+1)=q\rk{1-\frac{k}{M}}\, ,
\end{equation}
with $p+q=1$. Furthermore, set $a=pM/q$.

Next, set $f(i)$ the expected time to reach the state of all ones, starting in state $i$, i.e., 
\begin{equation}
    f(i)=\E_i\ek{H_M}\, ,
\end{equation}
where $H_k=\inf\gk{t\ge 0\colon X_t=k}$ is the first hitting time of state $k$.
\begin{lemma}
    We have that
    \begin{equation}
        f(M-i)=\frac{M}{i!i\Gamma(a+1)q}\sum_{k=0}^{i-1}\frac{\Gamma(a+i-k)}{\Gamma(i-k+1)}\, .
    \end{equation}
    In particular
    \begin{equation}\label{eq:313251}
       f(0)=\E_0[H_M]=\frac{1}{q}\frac{\Gamma\rk{M+1+a}}{\Gamma(a+1)\Gamma(M+1)}\, .
    \end{equation}
\end{lemma}
\begin{proof}
Note that the following recursion holds for $f(i)$: using the Markov property of the chain and the linearity of the expectation, we get for $\in \gk{0,\ldots, M-1}$
    \begin{equation}\label{eq:recc}
\begin{split}
    f(M)&=0\, ,\\
    f(i)&=1+pf(0)+q\tfrac{i}{M}f(i)+q\rk{1-\tfrac{i}{M}}f(i+1)\, .
\end{split}
\end{equation}
If we make the ansatz
\begin{equation}\label{eq:Ansatz}
    1+pf(0)=b_{k-1}+pf(k)\, ,
\end{equation}
we can quickly derive the following facts:
\begin{enumerate}
    \item The recursion for $i=1$ can be rewritten as 
    \begin{equation}
        f(0)=\frac{1}{q}+f(1)\, .
    \end{equation}
    Multiplying with $p$ and adding $1$ yields $b_0=\frac{1}{q}$.
    \item If we plug Eq.~\eqref{eq:Ansatz} into Eq.~\eqref{eq:recc}, we obtain
    \begin{equation}
        f(k)=b_{k-1}+pf(k)+q\tfrac{k}{M}f(k)+q\rk{1-\tfrac{k}{M}}f(k+1)\, ,
    \end{equation}
    which is equivalent to
    \begin{equation}
        f(k)=\frac{b_{k-1}}{q\rk{1-\tfrac{k}{M}}}+f(k+1)\, .
    \end{equation}
    Multiplying with $p$ and adding $b_{k-1}$ yields the recursion
    \begin{equation}\label{eq:recbk}
        b_{k}=b_{k-1}+b_{k-1}\frac{p}{q\rk{1-\tfrac{k}{M}}}\, .
    \end{equation}
    \item The recursion ends with $1+pf(0)=b_{M-1}$\, .
\end{enumerate}
Eq.~\eqref{eq:recbk} can be rewritten as $b_{k}=b_{k-1}\rk{1+\tfrac{pM}{q(M-k)}}$. It is a linear multiplicative recursion in $(b_k)_k$ and hence solvable. We hence obtain that (with $a=pM/q$)
\begin{equation}\label{eq:valofhelp}
    1+pf(0)=\prod_{k=1}^M\rk{1+\frac{p}{q}\frac{M}{k}}=\frac{\prod_{k=1}^M\rk{k+\tfrac{Mp}{q}}}{M!}=\frac{\Gamma\rk{M+1+a}}{\Gamma(a+1)\Gamma(M+1)}\, .
\end{equation}
This immediately implies
\begin{equation}\label{eq:f0zerasympt}
    f(0)=\frac{\Gamma\rk{M+1+a}}{p\Gamma(a+1)\Gamma(M+1)}-\frac{1}{p}\sim \frac{(M+1)^aM}{a\Gamma(a+1)q}\sim \frac{M^{a+1}}{\Gamma(a+1)a}\, .
\end{equation}
Now, Eq.~\eqref{eq:recc} gives that
\begin{equation}
    f(M-i)=M\frac{1+pf(0)}{pM+qi}+\frac{qi}{pM+qi}f(M-i+1)\, .
\end{equation}
Inserting our result from Eq.~\eqref{eq:valofhelp} into the recursion gives that
\begin{equation}
    f(M-i)=M(1+pf(0))i!\sum_{k=0}^{i-1}\frac{q^{k}}{(i-k)!\prod_{j=0}^{k}\rk{pM+(i-j)q}}\, .
\end{equation}
Note that
\begin{equation}
    \prod_{j=0}^{k}\rk{pM+(i-j)q}=q^{k+1}\frac{\Gamma(a+i+1)}{\Gamma(a+i-k)}\, ,
\end{equation}
and hence
\begin{equation}
    f(M-i)=\frac{\Gamma(M+1+a)\Gamma(i+1)}{\Gamma(a+i+1)\Gamma(a+1)\Gamma(M)}\sum_{k=0}^{i-1}\frac{\Gamma(a+i-k)}{\Gamma(i-k+1)} =\frac{M}{i!i\Gamma(a+1)}\sum_{k=0}^{i-1}\frac{\Gamma(a+i-k)}{\Gamma(i-k+1)}\, .
\end{equation}
This concludes the proof.
\end{proof}
Next, we derive the asymptotics for the expressions in the previous lemma.
\begin{corollary}\label{cor:singleCol}
    Assume now that $a=pM/q$ remains constant and $M\to\infty$. We then obtain
    \begin{equation}
        f(0)=\E_0\ek{H_M}\sim \frac{M^{a+1}}{\Gamma(a+1)a}\, ,
    \end{equation}
    as well as
\begin{equation}\label{eq:montf}
    \frac{(1+o(1))}{(1/a+1)}\le \frac{f(i)}{f(0)}\le 1\, .
\end{equation}
\end{corollary}
\begin{proof}
    The first relation was shown in Eq. \eqref{eq:f0zerasympt}. 
    
    For the second result, note that
    \begin{equation}
        f(0)=\E_0\ek{H_M}\ge \E_i\ek{H_M}=f(i)\, ,
    \end{equation}
    by stochastic coupling: indeed, to reach $M$ started from $0$, the walker needs to pass through $i$. Hence, the walker needs more time as compared to starting from $i$. See \cite{lindvall2002lectures} for an introduction to coupling/stochastic domination. Furthermore, note that Eq.~\eqref{eq:recc} implies that
    \begin{equation}
        f(M-1)\rk{p+\tfrac{q}{M}}=1+pf(0)\Longleftrightarrow f(M-1)=\frac{f(0)}{1+1/a}+\frac{1}{p+\tfrac{q}{M}}\, .
    \end{equation}
    The second statement then follows from the asymptotics of $f(0)$.
\end{proof}
Next, we prove an upper bound on the variance, in order to carry out a second moment argument later.
\begin{lemma}
    There exists a constant $C>0$ independent of $M$, such that
    \begin{equation}
        \Var_i\rk{H_M}\le C\E_i\ek{H_M}\, .
    \end{equation}
\end{lemma}
\begin{proof}
    Note that for $g(i)=\E_i\ek{H_M^2}$, we have that
\begin{align}
    g(i)&=1+2\rk{pf(0)+\tfrac{qi}{M}f(i)+q\rk{1-\tfrac{i}{M}}f(i+1)}+pg(0)+\tfrac{qi}{M}g(i)+q\rk{1-\tfrac{i}{M}}g(i+1)\\
    &=2f(i)-1+pg(0)+\tfrac{qi}{M}g(i)+q\rk{1-\tfrac{i}{M}}g(i+1)\, .
\end{align}
Hence, the variance $h(i)$ with $h(i)=g(i)-f(i)^2$ satisfies
\begin{equation}
    h(i)=p\rk{g(0)-f(i)^2}+\tfrac{qi}{M}h(i)+q\rk{1-\tfrac{i}{M}}h(i+1)\, .
\end{equation}
However, $g(0)-f(i)^2\le g(0)-f(0)^2=h(0)$, by Eq.~\eqref{eq:montf}. Define
\begin{equation}
    \tilde{h}(i)=p\tilde{h}(0)+\tfrac{qi}{M}\tilde{h}(i)+q\rk{1-\tfrac{i}{M}}\tilde{h}(i+1)\, .
\end{equation}
By monotonicity, we have that $h(i)\le \tilde{h}(i)$. Note that the recursion for $\rk{\tilde{h}(i)}_i$ is the same linear recursion as it was the for the expected value $f(i)_i$. We conclude that
\begin{equation}
    \Var(H_M)\le C \E\ek{H_M}\, .
\end{equation}
This finishes the proof.
\end{proof}
\begin{proof}[Proof of Theorem \ref{thm:singleCol}]
    The first statement of Theorem \ref{thm:singleCol} has been shown in Corollary \ref{cor:singleCol}.
    
    For the second statement, observe that the second moment computation from the previous Lemma gives
    \begin{equation}
        \P_i\rk{\abs{H_M-\E_i\ek{H_M}}>k}=\Ocal\rk{\frac{M^{a+1}}{k^2}}\, .
    \end{equation}

    The third statement follows directly from Eq.~\eqref{eq:invariantDist}.
    This concludes the proof.
\end{proof}
\subsection{The multicolumn model, PAI switched off}
Write $A\in \R^{M\times N}$, i.e., $M$-rows and $N$ columns. Write $Z_j$ for the transformation which sets the the $j$-th row to zero (when multiplied from the left):
\begin{equation}
    Z_j\in\R^{N\times M}\quad\text{with the i-th row of }Z_j\text{ is given by}\quad e_i^T(1-\delta_i(j))\, ,
\end{equation}
 where $e_i$ is the $i$-th unit vector in $\R^M$. Write $O_j$ for the matrix with $j$-th row equal to ones and the rest to zero:
 \begin{equation}
     O_j\in\R^{M\times N}\quad\text{where}\quad O_j(a,b)=\delta_j(a)\, .
 \end{equation}
Write $D_i\in\R^{N\times M}$ for the deletion of the $i$-th column (when multiplied from the right), i.e., the $k$-th column of $D_i$ is given by
\begin{equation}
    e_k(1-\delta_i(k))\, .
\end{equation}
We then write for the two actions on the current state: $H_j$ sets the $j$-th row equal to 1 (by deleting it first by the $Z_j$-left multiplication and then adding it through $O_j$) and $ W_i$ deletes the $i$-th row (using $D_i$)
\begin{equation}
    H_j(A)=Z_jA+O_j\quad \text{and}\quad W_i(A)=AD_i\, .
\end{equation}
We can the restate our evolution as
\begin{equation}
    \delta_A\mapsto \sum_{i=1}^N\frac{p}{N}\delta_{W_i(A)}+\sum_{j=1}^M \frac{q}{M}\delta_{H_j(A)}\, .
\end{equation}
We can also generalize the $O_j$'s to $O_S$, where $S$ is the set of $1$-rows and the $D_i$'s to $D_R$ where $R$ is the set of unit vectors. Note that
\begin{equation}
    H_{S_1}\circ H_{S_2}=H_{S_1\cup S_2}\quad\text{and}\quad W_{R_1}\circ W_{R_2}=W_{R_1\cup R_2}\, .
\end{equation}
A quick calculation shows that
\begin{equation}
    O_SD_R(i,j)=\delta_S(i)\delta_{R^c}(j)\, .
\end{equation}
Define $B(A,B)=O_AD_B$ and abbreviate the added rows and deleted columns as
\begin{equation}
    \overline{S_i}=\bigcup_{j=i}^TS_j\quad\text{and}\quad  \overline{R_i}=\bigcup_{j=i}^TR_j
\end{equation}
Performing a sequence of deletions $R_1,\ldots,R_T$ and additions $S_1,\ldots,S_T$ then leads to
\begin{equation}\label{eq:4125}
    W_{R_T}\circ H_{S_T}\circ\ldots\circ W_{R_1}\circ H_{S_1}(A)=Z_{\overline{S_1}}AD_{\overline{R_i}}+\sum_{i=1}^T B\rk{S_i\setminus\overline{S_{i+1}},\overline{R_i}}\, .
\end{equation}
This representation is beautiful, because we can infer the following construction for the invariant state:
\begin{enumerate}
    \item In each step, select \texttt{row} with probability $q$ and \texttt{column} with probability $p$.
    \item Depending on step 1, select uniformly at random a row in $\gk{1,\ldots, M}$ (resp. column in $\gk{1,\ldots, N}$).
    \item If \texttt{column} was selected, add this to the list of forbidden columns.
    \item If \texttt{row} was selected, switch the entries in that row to 1, if they are not in a forbidden column.
    \item Continue until there are no allowed columns left or you have selected all rows.
\end{enumerate}
\begin{lemma}\label{lem:invDist}
    The above procedure samples the from the invariant distribution.
\end{lemma}
\begin{proof}
    Letting the chain evolve for a sufficiently long time erases the dependence on the initial condition as $Z_{\overline{S_1}}AD_{\overline{R_i}}\equiv 0$ eventually. We see that reversing the order of summation in Eq. \eqref{eq:4125} is precisely the procedure outlined above. Recall that the probability that an exponential random clock with parameter $q$ rings before an independent clock with parameter $p$ is exactly given by $q$, as we assumed $p+q=1$.
\end{proof}
This is quite convenient, because we can estimate the probability that a column is equal to only ones.
\begin{lemma}
    Let $\tau$ be the time such that above procedure has selected each row at least once. We then have that for every $\e_M>\log^{-1/2}(M)$
    \begin{equation}
        \P\rk{\abs{\tau-\frac{M\log(M)}{q}}>\e_MM\log(M)}=o(1)\, .
    \end{equation}
    This implies a sharp transition on the $M\log(M)$ scale.
\end{lemma}
\begin{proof}
    Let $R(n)$ be the number of times we have selected \texttt{rows} after $n$ selections. A first moment argument together with a concentration inequality for Binomial random variables shows that for $M_2>M$
    \begin{equation}\label{eq:approx}
        \P\rk{\sup_{m\in \ek{M,M_2}}\abs{R(m)-mq}>\alpha}=\Ocal\rk{ M_2\ex^{-c\alpha^2/M}}\, .
    \end{equation}
    Next recall the following classic facts from the coupon collector's problem~\cite[Sec 3.3.1]{mitzenmacher2017probability}: if we are drawing from $n$ urns independently, it will take $n\log(n)$ steps until we have select each urn at least once. To be more precise, let $\sigma$ be the first time all urns were selected. We have that (after some application of the exponential Chebyshev's inequality)
    \begin{equation}\label{eq:uppertails313251}
        \P\rk{\sigma<n\log(n)-c n}\le \ex^{-3c^2/\pi^2}\quad\text{and}\quad \P\rk{\sigma>n\log(n)+c n}\le \ex^{-c}\, ,
    \end{equation}
    for $c>0$ larger than some finite constant. Hence, we can use Eq.~\eqref{eq:approx}, substituting $n=\tfrac{M}{q}$ and prove 
    \begin{equation}
        \P\rk{\abs{\tau-\frac{M\log M}{q}}>M\log^{1/2}(M)}=\Ocal\rk{\ex^{-\log^{1/2}(M)}}=o(1)\, ,
    \end{equation}
    which readily implies the lemma.
\end{proof}
We can use this to give an estimate of the number of columns equal to 1. However, before we need a technical lemma.
\begin{lemma}\label{lem:technical}
    Let $T$ be the collection time in the coupon collector's problem with $M$ coupons, where we draw according to an exponential random clock with parameter $q$. We then have that
    \begin{equation}
        \E\ek{\ex^{-\alpha T}}=\frac{\Gamma\rk{M+1}\Gamma\rk{1+M\alpha/q}}{\Gamma\rk{M+1+M\alpha/q}}\sim \frac{\Gamma(1+M\alpha/q)}{M^{M\alpha/q}}\, ,
    \end{equation}
    where asymptotics hold given that $\alpha/q$ does not vary with $M$.
\end{lemma}
\begin{proof}
    Recall that $T=\sum_{i=0}^{M-1}T_i$, where $T_i$ is exponential with parameter $p_i=q(1-i/M)$. We hence have that
    \begin{equation}
        \E\ek{\ex^{-\alpha T}}=\prod_{i=0}^{M-1}\frac{p_i}{p_i+\alpha}=\prod_{i=0}^{M-1}\frac{M-i}{M-i+M\alpha/q}=\frac{\Gamma\rk{M+1}\Gamma\rk{1+M\alpha/q}}{\Gamma\rk{M+1+M\alpha/q}}\, .
    \end{equation}
   The result now follows from Stirlings formula, as before.
\end{proof}
We are now ready to specify the expected number of columns equal to ones.
\begin{lemma}\label{lem:multColDen}
    Let $\chi_i=\1\gk{\textnormal{column }i\equiv 1}$. Then, if $M,N\to\infty$ and $M/N$ converges to a constant and $b=\tfrac{Mp}{qN}$ remains fixed, we have
    \begin{equation}
        \E_\pi\ek{\sum_i \chi_i}\sim  N  \frac{\Gamma(1+b)}{M^{b}}\, .
    \end{equation}
\end{lemma}
\begin{proof}
    We have that
    \begin{equation}
        \P\rk{\chi_1=1}=\P\rk{\tau<H_1}\, ,
    \end{equation}
    where $H_1$ is the first time we have selected column 1. A naive approximation $\tau\approx M\log(M)/q$ does not work here, as the upper tails in Eq.~\eqref{eq:uppertails313251} are too heavy. However, this can be rectified using Lemma \ref{lem:technical},as shown below:

    We expand
        \begin{equation}\label{eq:multl}
        \P\rk{\tau<H_1}=\int_0^\infty \P\rk{k<H_1|\tau= t}\P\rk{\tau=\d t}\d t=\int_0^\infty \ex^{-t p/N}\P\rk{\tau=\d t}\d t\, ,
    \end{equation}
    % \begin{multline}\label{eq:multl}
    %     \P\rk{\tau<H_1}=\Ocal\rk{\ex^{-\log^{4/3}(M)}}+\sum_{k\ge M^-}\P\rk{k<H_1|\tau=k}\P\rk{\tau=k}\\
    %     =\Ocal\rk{\ex^{-\log^{4/3}(M)}}+\sum_{k\ge M^-}\E_{(k,p)}\ek{\rk{1-\frac{1}{N}}^X}\P\rk{\tau=k}\, ,
    % \end{multline}
     since the exponential clocks are independent. %However, the above expectation is explicitly computable and equals
    % \begin{equation}
    %     \E\ek{\rk{1-\frac{1}{N}}^X}=\exp\rk{k\log\rk{1-\tfrac{p}{N}}}\, .
    % \end{equation}
    We can now use Lemma \ref{lem:technical} in conjuction with Eq.\eqref{eq:multl} to conclude that
    \begin{equation}
        \P\rk{\tau<H_1}=\frac{\Gamma\rk{M+1}\Gamma\rk{1+Mp/(Nq)}}{\Gamma\rk{M+1+Mp/(Nq)}}\sim \frac{\Gamma(1+Mp/Nq)}{M^{Mp/Nq}}\, .
    \end{equation}
    The result then follows from the linearity of the expectation as $ \E_\pi\ek{\sum_i \chi_i}=N\P\rk{\chi_1=1}$.
\end{proof}
We can now prove Theorem \ref{thm:multCol}.
\begin{proof}[Proof of Theorem \ref{thm:multCol}]
    Now that that the second statement (Eq.~\eqref{eq:RationPAIoff}) has been proven already in Lemma \ref{lem:multColDen}. It remains to prove bounds on the transition time, Eq.~\eqref{eq:transTimePAIoff}.

    First of all, note that on account of Eq. \eqref{eq:uppertails313251} and the union bound, we have that 
    \begin{equation}\label{eq:71520251}
        \P\rk{\exists c\in\gk{1,\ldots, N}\text{ with }c\equiv 1 \text{ at time }t}\, ,
    \end{equation}
    decays faster than any polynomial in $N$ and $M$ for all $t\le \tfrac{M\log(M)}{q}-M\log^{2/3}(M)$. 
    
    On the other hand, at time $t=\tfrac{M\log(M)}{q}+M\log^{2/3}(M)$, with probability $1-\Ocal\rk{\ex^{-\log(M)^{2/3}}}$, we will have chosen every row at least once. Hence on that event, using the representation of the evolution given in Eq. \eqref{eq:4125}, we see that the dependence on the initial $A\equiv 0$ has already been lost. This implies that we are in steady state (time-reversal!) and hence see $N\frac{\Gamma\rk{b+1}}{(M/q)^{b}}$ columns of all ones.
\end{proof}
\subsection{The multicolumn model, PAI switched on}
With the easier case solved, we now explain the adaptations needed for the general case. We begin with a description of the steady state.
\begin{lemma}\label{LemPAIon}
    The following algorithm samples the steady-state distribution of the (jump-chain) of the model with PAI switched on: start with the zero matrix. Evolve time in discrete steps. At each step, 
    \begin{enumerate}
        \item with probability $p/(1+N\l_m)$, select a uniformly chosen column to the list of forbidden columns,
        \item with probability $q/(1+N\l_m)$, turn the entries of a uniformly selected row into ones, given that the entry is not located in a forbidden row,
        \item with probability $\l_m N/(1+N\l_m)$, turn a uniformly selected entry of the matrix into a one, given that the entry is not a forbidden column.
    \end{enumerate}
\end{lemma}
The argument for the proof of Lemma \ref{LemPAIon} follows the same lines as that of Lemma \ref{lem:invDist}, based on time-reversal.

Recall the new parameter $\widetilde{q}=q+\lambda_m$. Denote $\widetilde{b}=\tfrac{\widetilde{p}M}{\widetilde{q}N}$ and assume that it remains constant.
Next, we analyze the time it takes to approach the steady state.
\begin{lemma}
    We have that for every $\e_M>\log^{-1/2}(M)$
    \begin{equation}
        \P\rk{\abs{\tau-\frac{M\log(M)}{\widetilde{q}}}>\e_MM\log(M)}=o(1)\, .
    \end{equation}
\end{lemma}
\begin{proof}
    Let $P_{N,k}$ be the probability that the coupon collector problem finishes in at most $k$-draws. It is well known that
    \begin{equation}
        P_{N,k}=\sum_{i=1}^n(-1)^{N-i}\binom{N}{i}\rk{\frac{i}{N}}^k\, .
    \end{equation}
    If we make the draws in continuous-time, we have that $P_{N,t}$ is equal to
    \begin{equation}\label{eq:71520252}
        P_{N,t}=\E_t\ek{P_{N,X}}=\ex^{-t}\sum_{i=1}^N\binom{N}{i}\ex^{t(i/N)}\rk{-1}^{N-i}=\rk{\rk{1-\ex^{-t/N}}^N-(-1)^N\ex^{-t}}\, .
    \end{equation}
    Note that the above can be approximated (for $t\gg N$) by $\exp\rk{-N\ex^{-t/N}(1+o(1))}$. Combining the reasoning from Eq.~\eqref{eq:71520251} with the estimate from Eq.~\eqref{eq:71520252}, we can conclude that at time time $\frac{M\log(M)}{\widetilde{q}}-\log(M)^{2/3}$, not a single column is identical to $(1,\ldots,1)^T$. By the time-reversal argument from the PAI-off proof, the upper bound of $\frac{M\log(M)}{\widetilde{q}}+\log(M)^{2/3}$ follows.
\end{proof}

Next, we give the number of components in steady-state. As the proof is similar to that of Lemma~\ref{lem:multColDen}, we leave it to the reader.
\begin{lemma}
      Recall that $\chi_i=\1\gk{\textnormal{column }i\equiv 1}$. Then, if $M,N\to\infty$ and $M/N$ converges to a constant, we have
    \begin{equation}
        \E_\pi\ek{\sum_i \chi_i}\sim  N\frac{\Gamma\rk{\widetilde{b}+1}}{ (M/\widetilde{q})^{\widetilde{b}}}\, .
    \end{equation}
\end{lemma}
With that, the proof of Theorem~\ref{thm:multCol} concludes.

\bibliographystyle{alpha}
\bibliography{reference.bib}
\end{document}